\documentclass{amsart}

\usepackage{hyperref}
\hypersetup{
    colorlinks = true,
    allcolors = blue,
}

\usepackage{mathtools}
\usepackage{amsmath, amsfonts, amssymb}
\usepackage{dsfont}
\usepackage{amsthm}
\usepackage{xfrac}
\usepackage{faktor}
\usepackage{enumerate}
\usepackage[dvipsnames]{xcolor}
\usepackage{tikz}
\usepackage{tikz-cd}
\usepackage{blkarray}
\usepackage[subnum]{cases}
\usepackage{multirow}
\usepackage{setspace}
\usepackage[normalem]{ulem}
\usepackage{soul}
\usepackage{caption}
\usepackage{fullpage}
\usepackage{mathrsfs}

\newtheorem{thm}{Theorem}[section]
\newtheorem{lemma}[thm]{Lemma}
\newtheorem{prop}[thm]{Proposition}
\newtheorem{cor}[thm]{Corollary}

\theoremstyle{plain} 
\newcommand{\thistheoremname}{}
\newtheorem*{genericthm}{\thistheoremname}

\theoremstyle{definition}
\newtheorem{example}[thm]{Example}
\newtheorem{defn}[thm]{Definition}
\newtheorem{remark}[thm]{Remark}

\newcommand{\be}{\mathbf{e}}

\newcommand{\BC}{\mathbb{C}}
\newcommand{\BR}{\mathbb{R}}
\newcommand{\BQ}{\mathbb{Q}}
\newcommand{\BZ}{\mathbb{Z}}
\newcommand{\BN}{\mathbb{N}}

\newcommand{\cB}{\mathcal{B}}
\newcommand{\cC}{\mathcal{C}}
\newcommand{\cF}{\mathcal{F}}
\newcommand{\cH}{\mathcal{H}}

\newcommand{\cS}{\mathcal{S}}

\newcommand{\fg}{\mathfrak{g}}

\newcommand{\fs}{\mathfrak{s}}
\newcommand{\ft}{\mathfrak{t}}
\newcommand{\fA}{\mathfrak{A}}

\newcommand{\fj}{\mathfrak{j}}

\renewcommand{\d}{\delta}
\newcommand{\e}{\varepsilon}
\newcommand{\z}{\zeta}
\renewcommand{\k}{\kappa}
\newcommand{\ld}{\lambda}

\newcommand{\s}{\sigma}

\newcommand*{\defeq}{\mathrel{\vcenter{\baselineskip0.5ex \lineskiplimit0pt
                     \hbox{\scriptsize.}\hbox{\scriptsize.}}}%
                     =}
\let\latexcirc=\circ
\newcommand{\ccirc}{\mathbin{\mathchoice
  {\xcirc\scriptstyle}
  {\xcirc\scriptstyle}
  {\xcirc\scriptscriptstyle}
  {\xcirc\scriptscriptstyle}
}}
\newcommand{\xcirc}[1]{\vcenter{\hbox{$#1\latexcirc$}}}
\let\circ\ccirc

\DeclarePairedDelimiter{\abs}{\lvert}{\rvert}%
\DeclarePairedDelimiter{\norm}{\lVert}{\rVert}%
\DeclarePairedDelimiter{\angleb}{\langle}{\rangle}%
\DeclareMathOperator{\id}{id}

\DeclareMathOperator*{\res}{res}

\renewcommand{\mod}{\operatorname{mod }}

\DeclareMathOperator{\Aut}{Aut}

\DeclareMathOperator{\Stab}{Stab}
\DeclareMathOperator{\GL}{GL}
\DeclareMathOperator{\PGL}{PGL}
\DeclareMathOperator{\U}{U}

\DeclareMathOperator{\ord}{ord}

\newcommand{\inv}{^{-1}}
\newcommand{\nth}{^{\mathrm{th}}}
\newcommand{\ol}{\overline}

\newcommand{\restr}[1]{\big\vert_{#1}}
\newcommand{\SL}{{\operatorname{SL}_2(\BZ)}}
\newcommand{\PSL}{{\operatorname{PSL}_2(\BZ)}}
\newcommand{\qsl}[1]{{\operatorname{SL}_2(\BZ/#1 \BZ)}}

\newcommand{\smallmat}[1]{
    \left[\begin{smallmatrix}#1\end{smallmatrix}\right]
}

\newcommand{\tf}{\tilde{f}}

\def\O{{\operatorname{O}}}
\def\ind{{\operatorname{Ind}}}
\def\res{{\operatorname{Res}}}
\def\Res{{\res}}
\def\Ind{{\ind}}
\def\trho{{\tilde{\rho}}}

\usepackage{braket}
\usepackage[skip=3pt, indent = 4ex]{parskip}

\title{On symmetric representations of $\SL$}
\author{Siu-Hung Ng} 
\address{Department of Mathematics, Louisiana State University, Baton Rouge, LA 70803.}
\email{rng@math.lsu.edu}
\thanks{The authors were partially supported by the NSF grant DMS 1664418.}
\author{Yilong Wang}
\address{Beijing Institute of Mathematical Sciences and Applications (BIMSA), Huairou, Beijing, China.}
\email{wyl@bimsa.cn}
\author{Samuel Wilson}
\address{Department of Mathematics, Louisiana State University, Baton Rouge, LA 70803.}
\email{swil311@lsu.edu}
\date{}
\begin{document}
\begin{abstract} We introduce the notions of symmetric and symmetrizable representations of $\SL$. The linear representations of $\SL$ arising from modular tensor categories are symmetric and have congruence kernel. Conversely, one may also reconstruct modular data from finite-dimensional symmetric, congruence representations of $\SL$. By investigating a $\BZ/2\BZ$-symmetry of some Weil representations  at prime power levels, we prove that all finite-dimensional congruence representations of $\SL$ are symmetrizable. We also provide examples of unsymmetrizable noncongruence representations of $\SL$ that are subrepresentations of a symmetric one.
\end{abstract}

\maketitle

\section{Introduction}
The group $\SL$ plays an integral role in the theory of modular forms. Representations of $\SL$ also appear naturally in rational conformal field theory (RCFT) and topological quantum field theory (TQFT).  In both of these theories, the representations arise from underlying modular tensor categories (MTC). Readers are referred to \cite{BK01, EGNO} for more details on modular tensor categories.  For their relations to RCFT, see \cite{MS89, Hu, Z}; for TQFT, see \cite{RT91, Tur10}.  MTCs also form the foundation for topological quantum computation and topological phases of matter, regarding which see \cite{LW04, Kit06, RW18}.

Associated to a modular tensor category $\cC$ is a pair of complex square matrices, $(S,T)$, called the modular data of $\cC$.  The group $\SL$ is generated by $\fs = \smallmat{0&1\\-1&0}$ and $\ft = \smallmat{1&1\\0&1}$, and the assignment $(\fs, \ft) \to (S^{-1},T)$ defines a projective representation\footnote{This convention of $\fs$ is adopted from \cite{Nobs1}; literature on MTCs usually uses the inverse.}.  This can be linearized to a unitary matrix representation $\rho$ where $\rho(\fs)$ is symmetric and $\rho(\ft)$ is diagonal. We call representations of $\SL$ with these properties \emph{symmetric}. Moreover, $\rho$ is \emph{congruence}, i.e.~has a congruence kernel \cite{NS10, DLN15}.

The family of \emph{pointed} MTCs, which can be built from finite abelian groups equipped with nondegenerate quadratic forms \cite{JS85, JS93}, are particularly relevant to this paper (see Example \ref{ex:mtc} and Section \ref{subsec:quadforms}). The projective representation arising from such a category coincides with the \emph{Weil representation} of the quadratic form; the study of Weil representations has a long history, including works such as \cite{Klo46,Weil64,Tan67}.

Given a congruence representation of $\SL$, it is natural to ask whether it can be realized by a MTC in this way, and if so, how to reconstruct the modular data.  It is clear that, for the representation to be realized, it is necessary for it to be \emph{symmetrizable}\,---\,that is, to admit a basis with respect to which it is symmetric.  Representations of $\SL$ that are not symmetrizable do exist (see Section \ref{subsec:examples}); however, our examples are all noncongruence representations.  The main result of this paper, Theorem \ref{thm:main}, is that every finite-dimensional congruence representation of $\SL$ is symmetrizable.

We prove this theorem by investigating irreducible representations of $\qsl{p^\ld}$ for primes $p$ and positive integers $\ld$, which were completely classified by Nobs--Wolfart using subrepresentations of Weil representations \cite{NW76}. The main thrust of the proof is the existence of a certain $\BZ/2\BZ$ symmetry, derived from an involutive automorphism, for each relevant quadratic form.  We show that the subspace associated to each irreducible subrepresentation is invariant under that symmetry; this implies the representation is symmetrizable.  Based on this proof, the authors have implemented a \texttt{GAP} package, \texttt{SL2Reps} \cite{SL2Reps1.0}, which automatically generates a symmetric basis for each irreducible congruence representation $\rho$ and outputs the corresponding matrices $\rho(\fs)$ and $\rho(\ft)$.  In fact, these symmetric, irreducible congruence representations are essential for the reconstruction of modular data from representations of $\SL$ \cite{NRWW}.  In special cases (see, for example, \cite[Sec.~5]{NWZ22} and \cite[Sec.~3]{Rank5}), $\rho(\fs)$ and $\rho(\ft)$ will completely determine the fusion rules of a potential MTC realizing $\rho$.

The paper is organized as follows. In Section \ref{sec:main}, we introduce symmetric and congruence representations of $\SL$ and provide some examples. In Section \ref{sec:symm}, we describe Weil representations in general and establish criteria for symmetrizability. Then, in Section \ref{sec:irreps}, we consider irreducible congruence representations of prime power level in detail and prove the main theorem. Finally, we give some applications of the result.

\section{Symmetric representations of \texorpdfstring{$\SL$}{}}\label{sec:main}

\subsection{Notation and definitions}\label{subsec:defn}

Let $\fs \defeq \smallmat{0&1\\-1&0}$ and $\ft \defeq \smallmat{1&1\\0&1}$, a choice of generators for the group $\SL$. For any real number $a \ge 0$, let $\sqrt{a}$ denote its nonnegative square root. The group of $n \times n$ unitary complex matrices is denoted by $\U(n)$. We write a complex number $u \in \U(1)$ of complex modulus 1 as $u = e^{ix}$ for some $x \in [\,0, 2\pi)$ and define $\sqrt{u} \defeq e^{ix/2}$. For any $r \in \BQ$, we write $\be(r) \defeq e^{2\pi i r}$, and for any positive integer $k$, we write $\zeta_k \defeq \be(1/k)$. In particular, $\sqrt{-1} = i = \be(1/4) = \zeta_4$ in our convention. Finally, we write $\big(\frac{k}{p}\big)$ for the Legendre symbol of $k$ mod $p$. All representations of $\SL$ considered are finite-dimensional over $\BC$.

\begin{defn}
    A unitary matrix representation $\rho: \SL \to \U(n)$ is called \emph{symmetric} if the following two conditions hold:
    \begin{itemize}
        \item $\rho(\fs)$ is symmetric;
        \item $\rho(\ft)$ is diagonal.
    \end{itemize}
    For any finite-dimensional Hilbert space $V$, a representation $\rho: \SL \to \GL(V)$ is called \emph{symmetrizable} if it is equivalent to a symmetric representation.  An equivalent condition is that $V$ admits an orthonormal basis, called a \emph{symmetric basis} for $\rho$, with respect to which the matrix presentation of $\rho$ is symmetric.
\end{defn}

\begin{remark}\label{rem:preservation}
\begin{itemize}
    \item[(i)] Any permutation of a symmetric basis for $\rho$ is also a symmetric basis for $\rho$.
    \item[(ii)] Symmetrizability is preserved under direct sum and tensor product of representations. Indeed, if $\rho_1$, $\rho_2$ admit symmetric bases $\cB_1$, $\cB_2$ respectively, then $\{(v_1, 0)\mid v_1\in\cB_1\}\cup\{(0,v_2)\mid v_2\in\cB_2\}$ and $\{v_1\otimes v_2 \mid v_1\in\cB_1, v_2\in\cB_2\}$ are symmetric bases for $\rho_1\oplus \rho_2$ and $\rho_1\otimes \rho_2$ respectively.
\end{itemize}
\end{remark}

\begin{defn}
    A finite-dimensional representation $\rho$ of $\SL$ is called \emph{congruence} of level $n$ if $\ker(\rho)$ is a congruence subgroup of level $n$.
\end{defn}

In particular, a congruence representation $\rho:\SL\to \GL(V)$ of level $n$ factors through $\qsl{n}$.  The level of $\rho$ is equal to the order of $\rho(\ft)$ \cite[Lem.~A.1]{DLN15}.

\subsection{Examples of symmetrizable and unsymmetrizable representations}\label{subsec:examples}

\begin{example}\label{ex:mtc}
As mentioned in the introduction, a family of representations of $\SL$ may be obtained from any modular tensor category as follows. Let $\cC$ be a modular tensor category\,---\,that is, a braided fusion category equipped with a ribbon structure whose braiding satisfies a nondegeneracy condition (we refer the reader to \cite{BK01, EGNO, Tur10} for details). The ribbon structure on $\cC$ induces a trace on endomorphisms, and by taking the traces of double braidings and twists of simple objects, one obtains the modular data $(S,T)$, a pair of complex matrices indexed by the (finite) set of isomorphism classes of simple objects in $\cC$. With respect to this natural basis, $S$ is symmetric and $T$ is diagonal \cite[Chap.~3]{BK01}.

Let $r$ be the number of isomorphism classes of simple objects of $\cC$. It is well-known (see, for example, \cite{Tur10, BK01}) that the assignment $(\fs, \ft) \to (S^{-1},T)$ defines a projective representation $\tilde{\rho}_\cC : \SL \to \PGL_r(\BC)$, which can be lifted to a linear representation of $\SL$ by scaling $S$ and $T$, and there are 12 distinct such lifts \cite[Thm.~7.1]{NS10}. Let $\rho_\cC:\SL\to\GL_r(\BC)$ be any of these linear lifts of $\tilde{\rho}_\cC$.  By the discussion in the previous paragraph, $\rho_\cC$ is a symmetric representation. Further, by \cite[Thm.~II]{DLN15}, $\ker(\rho_\cC)$ is a congruence subgroup of $\SL$. Thus, $\rho_\cC$ is a symmetric congruence representation of $\SL$.
\end{example}

For the next example, we will use the following lemma.

\begin{lemma} \label{lem:non-sym}
Let $\rho:\SL\to\U(n)$ be a representation such that $\rho(\ft)$ is diagonal. Denote $s = \rho(\fs)$ and $t = \rho(\ft)$.  If $\rho$ is symmetrizable, then for any three indices $j,k,\ell \in \{1,\dots,n\}$ such that the eigenvalues $t_{j,j}$, $t_{k,k}$, and $t_{\ell,\ell}$ of $\rho(\ft)$ all have multiplicity 1, we have $s_{j,k} \cdot s_{k,\ell} \cdot s_{\ell,j} = s_{j,\ell} \cdot s_{\ell,k} \cdot s_{k,j}$.
\end{lemma}
\begin{proof}
For $A \in \GL_n(\BC)$ and $B \in \operatorname{M}_n(\BC)$, we write $B^A \defeq A^{-1}BA$. Suppose that $\rho$ is symmetrizable. Then there exists a unitary matrix $A$ such that $s^A$ is symmetric and $t^A$ is diagonal. As $t^A$ and $t$ are diagonal and have the same eigenvalues, there is a permutation matrix $P$ such that $t^{AP} = t$. Denote $U \defeq AP$. Then $Ut = tU$, so $U$ is a unitary block-diagonal matrix; the blocks correspond to the distinct eigenvalues of $t$, and each has size equal to the corresponding multiplicity. In particular, since $t_{j,j}$ is of multiplicity 1, there must be some $u_j \in \U(1)$ such that $U_{i,j} = \delta_{ij} u_j$ for all $1 \leq i \leq n$. The same holds for $k$ and $\ell$.

Now, $s^U$ is symmetric. So, $\ol{u}_j u_k  s_{j,k} = (s^U)_{j,k}= (s^U)_{k,j} = \ol{u}_k u_j s_{k,j}$ and hence $u_k^2 s_{j,k} = u_j^2 s_{k,j}$. Similarly, we have $u_\ell^2 s_{k,\ell}  = u_k^2 s_{\ell,k}$ and $u_j^2 s_{\ell, j}  = u_\ell^2 s_{j, \ell}$, and the statement follows immediately.
\end{proof}

\begin{example}\label{ex:noncongruence}
Following \cite{FF20}, we consider the four homomorphisms from $\SL$ to the permutation group $S_7$ shown in Table \ref{tab:noncong}.

\begin{table}[ht]
    \centering
    $\def\arraystretch{1.2}
    \begin{array}{|c|c|c|c|c|}
    \hline
     & {\phi_1} & {\phi_2} & {\phi_3} & {\phi_4} \\
    \hline\hline
    \fs & (12)(34)(56) & (12)(34)(56) & (12)(34)(67) & (12)(34)(67) \\
    \hline
    \ft & (1245)(367) & (12475)(36) & (124735) & (125473)\\
    \hline
    \abs{\operatorname{Im}(\phi_j)} & 7! & 7! & 42 & 42\\
    \hline
    n_k  & 12 & 10 & 6 & 6\\ 
    \hline
    \abs{\qsl{n_k}}& 2^7 3^2 & 2^4 3^2 5 & 2^4 3^2 & 2^4 3^2\\
    \hline
    \end{array}$
    \caption{Four noncongruence permutation representations of $\SL$.}
    \label{tab:noncong}
\end{table}

Notice that, for each $k$, $\abs{\operatorname{Im}(\phi_k)}$ does not divide $\abs{\qsl{n_k}}$, where $n_k = \ord(\phi_k(\ft))$. Therefore, the homomorphism $\phi_k$ has a noncongruence kernel. Further, let $\rho: S_7 \to \U(7)$ be the permutation representation of $S_7$ on $V=\BC^7$, and let $\{e_j\}_{i=1}^7$ denote the standard basis of $V$. Since $\rho$ is faithful, $\ker(\phi_k) = \ker(\rho\circ\phi_k)$, so $\rho\circ\phi_k$ is a noncongruence representation of $\SL$.  For brevity, we write $\rho_k \defeq \rho \circ \phi_k$, and view $\ft$ as a permutation in $S_7$, namely $\phi_k(\ft)$ for the relevant choice of $k$.

The representations $\rho_1$ and $\rho_2$ thus constructed are symmetrizable. It is clear that the set
\[\scriptstyle
\cB_1 \defeq \Big\{
\sum\limits_{a=0}^{3} e_{\ft^a(1)},\,
\ol \zeta_8 \sum\limits_{a=0}^{3} i^{a} e_{\ft^{a}(1)},\,
-i \sum\limits_{a=0}^{3} (-1)^{a} e_{\ft^{a}(1)},\,
\zeta_8\sum\limits_{a=0}^{3} (-i)^{a} e_{\ft^{a}(1)},\,
\sum\limits_{a=0}^{2} e_{\ft^{a}(3)},\,
\sum\limits_{a=0}^{2} \z_3^a e_{\ft^{a}(3)},\,
\sum\limits_{a=0}^{2} \z_3^{2a} e_{\ft^{a}(3)}
\Big\}\]
is an orthogonal eigenbasis for $\rho_1(\ft)$. One can check directly that the normalization of $\cB_1$ is a symmetric basis for $\rho_1$. Similarly, the normalization of the orthogonal basis
\[\scriptstyle\Big\{
\sum\limits_{a = 0}^{4} e_{\ft^a(1)},\,
-\zeta_5^2\sum\limits_{a = 0}^{4} \z_5^a e_{\ft^a(1)},\, 
\ol\zeta_5\sum\limits_{a = 0}^{4} \z_5^{2a} e_{\ft^a(1)},\,
\zeta_5 \sum\limits_{a = 0}^{4} \z_5^{3a} e_{\ft^a(1)},\,
-\zeta_5^3\sum\limits_{a = 0}^{4} \ol\zeta_5^{a} e_{\ft^a(1)},\, \sum\limits_{a=0}^{1}e_{\ft^a(3)},\,
-i\sum\limits_{a=0}^{1}(-1)^a e_{\ft^a(3)}
\Big\}\]
is a symmetric basis for $\rho_2$.

On the other hand, the representations $\rho_3$ and $\rho_4$ are not symmetrizable. Consider the ordered eigenbasis $\cB_3 \defeq \{v_1, \dots, v_7\}$ for $\rho_3(\ft)$  given by
\[
    v_1 \defeq e_6, \qquad
    v_i \defeq \frac{1}{\sqrt{6}}\sum_{a=0}^5 \zeta_6^{(i-2)a} e_{\ft^a(1)} \quad \text{for } i \in \{2,\dots,7\}.
\]
Let $s= \rho_3(\fs)$ and $s(v_j) = \sum_{j=1}^7 s_{i,j} v_i$ for $i, j \in \{ 1, \dots, 7 \}$. The eigenvectors $v_3, v_4, v_5$ of $\rho_3(\ft)$ have eigenvalues of multiplicity 1, and
\[
    s_{3,4}=\ol{s_{4,3}} = \frac{5-\sqrt{3}i}{12}\,, \qquad
    s_{3,5} = s_{4,5} = \ol{s_{5,3}} = \ol{s_{5,4}} = -\frac{2+\sqrt{3}i}{6}\,.
\]
In particular, $s_{3,4} \cdot s_{4,5} \cdot s_{5,3} \neq s_{3,5} \cdot s_{5,4} \cdot s_{4,3}$. It follows from Lemma \ref{lem:non-sym} that $\rho_3$ is not symmetrizable. The same argument may be applied to show that $\rho_4$ is not symmetrizable either.
\end{example}

A symmetrizable representation can be obtained from any real orthogonal representation of $\PSL$ via the induction functor $\Ind_{\SL}^{\GL_{2}(\BZ)}$. The group $\GL_{2}(\BZ)$ is a semidirect product: $\GL_2(\BZ) = \SL \rtimes \angleb{\fj}$, where $\fj = \smallmat{1 & 0 \\ 0 & -1}$.  Conjugation by $\fj$ defines an automorphism $\s$ of $\SL$, and
\[
    \s(\fs) = \fs^{-1}\qquad \text{and}\qquad \s(\ft) = \ft^{-1}\,.  
\]
For any representation $\rho:\SL \to \GL_n(\BC)$, let $\tilde{\rho} \defeq \Res_{\SL} \Ind_{\SL}^{\GL_{2}(\BZ)} \rho$, the restricted induced representation of $\rho$. As per \cite{Mac51}, we have $\trho \cong \rho \oplus (\rho\circ \s)$. In particular, for any $x,y \in V$,
\[
    \trho(\fs)(x,y) = (\rho(\fs)x,\rho(\fs)^{-1}y) \qquad \text{and} \qquad \trho(\ft)(x,y) = (\rho(\ft)x,\rho(\ft)^{-1}y)\,.
\]
To simplify notation, given any representation $\eta: \PSL\to\GL_n(\BC)$, we again use $\eta$ to denote the representation $\SL \to \PSL \xrightarrow{\eta} \GL_n(\BC)$.

\begin{prop}\label{prop:ind}
Let $\rho$ be any representation $\PSL \to \O(n)$, where $\O(n) = \U(n) \cap \GL_{n}(\BR)$ is the group of orthogonal matrices. Then $\trho$ is symmetrizable.
\end{prop}

\begin{proof}
By assumption, $\rho(\fs) = \rho(\fs)^\top$. Hence, for any $x, y \in \BC^n$, 
\[
    \angleb{\rho(\fs) \ol{x}, \ol{y}} = \ol{\angleb{\rho(\fs)x, y}} = \angleb{y, \rho(\fs)x} = \angleb{\rho(\fs)y, x}.
\]
Consequently, we have
\begin{equation}\label{eq:2.7-1}
    \angleb{\trho(\fs)(x,\ol{x}), (y, \ol{y})} = \angleb{\rho(\fs)x, y} + \angleb{\rho(\fs)\ol{x}, \ol{y}} = \angleb{\rho(\fs) \ol y, \ol x} + \angleb{\rho(\fs)y, x} = \angleb{\trho(\fs)(y,\ol{y}), (x, \ol{x})}\,.
\end{equation}

Now we construct a symmetric basis for $\trho$. Since $\rho(\ft)$ is orthogonal, and in particular normal, there exists an orthonormal eigenbasis for $\rho(\ft)$, denoted by $\{v_j\}_{j = 1}^{n}$.  Let $\ld_j \in \U(1)$ be the eigenvalue for $v_j$. Then $\rho(\ft)^{-1}\ol{v_j} = \ol{\rho(\ft)^{-1} v_j} = \ld_j \ol{v_j}$. As such,
\[
\cB_{\trho} \defeq \Big\{\frac{1}{\sqrt{2}}\Big(\sqrt{\e}v_j, \ol{\sqrt{\e}v_j}\Big)~\Big\vert~\e \in \{\pm 1\}, 1 \le i \le n \Big\}
\]
is an orthonormal eigenbasis for $\trho(\ft)$. Finally, $\trho(\fs)$ is symmetric with respect to $\cB_{\trho}$ by \eqref{eq:2.7-1}.
\end{proof}

\begin{example}\label{ex:nonsymmsub}
Subrepresentations of a symmetric representation may fail to be symmetrizable. Indeed, $\rho_3$ in Example \ref{ex:noncongruence} fulfils the condition of Proposition \ref{prop:ind}, so $\tilde\rho_3$ is symmetric. However, $\trho_3$ contains $\rho_3$, which is not symmetrizable, as a subrepresentation.  In fact, $\rho_3 \circ \s$ is not symmetrizable either (by a similar argument to that in Example \ref{ex:noncongruence}). Notably, since $\rho: S_7 \to \U(7)$ is faithful, $\ker(\trho_3) = \ker(\phi_3) \cap \s(\ker (\phi_3))$ is not a congruence subgroup of $\SL$. 
\end{example}

If a subrepresentation of a symmetric representation admits additional symmetry, then it is symmetrizable. The following lemma will be used in the subsequent sections.

\begin{lemma}\label{lem:sym-1}
Let $\eta: \SL \to \U(n)$ be a symmetric representation. Suppose $U\in\U(n)$ commutes with $\eta(\fg)$ for all $\fg \in \SL$. Let $\varphi(x) = Ux$ and $\ol \varphi(x) = \ol{U x}$ for $x \in \BC^n$; note that $\ol\varphi$ is an antilinear operator. Then:
\begin{itemize}
\item[(i)] 
for any $x, y \in \BC^n$, we have $\angleb{\eta(\fs)x,\,y} = \angleb{\eta(\fs)\ol{\varphi}(y),\,\ol{\varphi}(x)}$.
\item[(ii)]
If  $\rho$ is a subrepresentation of $\eta$ and there exists an orthonormal eigenbasis $\cS$ for $\rho(\ft)$ that is fixed by $\ol\varphi$ pointwisely, then $\cS$ is a symmetric basis for $\rho$.
\end{itemize}
\end{lemma}
\begin{proof}
Since $\eta(\fs)$ is symmetric, $\ol{\eta(\fs)} = \eta(\fs)^{-1}$, which implies $\angleb{\eta(\fs)\ol{x},\,\ol{y}} = \ol{\angleb{\eta(\fs)^{-1}x, y}} = \angleb{\eta(\fs)y,\,x}$ for any $x, y \in \BC^n$. As a result, we have
\begin{equation*}
\angleb{\eta(\fs)\ol{\varphi}(y),\,\ol{\varphi}(x)}  = \angleb{\eta(\fs)\varphi(x),\,\varphi(y)} =\angleb{\varphi(\eta(\fs)x),\,\varphi(y)} = \angleb{\eta(\fs)x,\,y}\,,
\end{equation*}
which proves (i).

By  (i), for any $x, y\in \cS$, the  matrix coefficients of $\rho(\fs)$ are given by
\[
    \rho(\fs)_{y,x} = \angleb{\eta(\fs)x, y} = \angleb{\eta(\fs)\ol{\varphi}(y), \ol{\varphi}(x)} = \angleb{\eta(\fs)y, x} = \rho(\fs)_{x,y}\,,
\]
which means $\rho(\fs)$ is symmetric with respect to $\cS$. Since $\cS$ is an eigenbasis for $\rho(\ft)$, $\rho$ is symmetric with respect to $\cS$.
\end{proof}

\subsection{Statement of the main results}
From the above examples, we can see that representations of $\SL$ can fail to be symmetrizable, and such representations cannot arise from any modular tensor category. However, our examples for this behavior are noncongruence representations, and hence are not very helpful in the study of MTCs: all $\SL$-representations coming from an MTC have to be congruence in the first place. Therefore, it is natural to ask if congruence representations can also fail to be symmetrizable. The main result of this paper is the following theorem.

\begin{thm}\label{thm:main}
Every finite-dimensional congruence representation of $\SL$ is symmetrizable.
\end{thm}
\begin{proof}
Let $\rho$ be a congruence representation of level $n$. Since $\rho$ factors through $\qsl{n}$, it decomposes into a direct sum of irreducible representations of $\qsl{n}$. If each of the irreducible components of $\rho$ is symmetrizable, then by Remark \ref{rem:preservation}, $\rho$ is also symmetrizable. Thus, we may assume without loss of generality that $\rho$ is irreducible.  Then, applying the Chinese remainder theorem and \cite[Thm.~3.2.10]{Ser77}, Theorem \ref{thm:main} follows from Proposition \ref{prop:main}.
\end{proof}

\begin{prop}\label{prop:main}
Let $p$ be a prime and $\ld$ be a positive integer. Every irreducible representation of $\qsl{p^\ld}$ is symmetrizable.
\end{prop}
The proof of Proposition \ref{prop:main} will be provided in Sections \ref{sec:symm} and \ref{sec:irreps}.

\section{Weil representations and symmetrizability}\label{sec:symm}

The irreducible representations of $\qsl{p^\ld}$ have been classified by Nobs and Wolfart \cite{NW76}, and all such representations can be built from subrepresentations of Weil representations (as detailed in Section \ref{sec:irreps}). In this section, we first define quadratic modules and Weil representations in general, then establish some criteria for the symmetrizability of subrepresentations thereof.

\subsection{Quadratic forms and Weil representations}\label{subsec:quadforms}
\begin{defn}\label{defn:weil}
Let $M$ be an additive abelian group.  A \emph{nondegenerate quadratic form} on $M$ is a function $Q : M \to \BQ/\BZ$ such that
\begin{enumerate}
    \item[(i)] $Q(-a) = Q(a)$ for all $a \in M$ and
    \item[(ii)] $B(a,b) \defeq Q(a+b) - Q(a) - Q(b)$ defines a nondegenerate bilinear map.
\end{enumerate}
The pair $(M,Q)$ is then called a (nondegenerate) \emph{quadratic module}. In this paper, all quadratic modules are assumed to be nondegenerate.
\end{defn}
Quadratic modules are closely related to \emph{pointed} modular categories, in which the isomorphism classes of simple objects form an abelian group under the tensor product (see, for example, \cite[Sec.~8]{ENO05}). Precisely: on the one hand, given any pointed modular category $\cC$, the group of isomorphism classes of simple objects, together with the function defined by their twists, forms a quadratic module; on the other hand, given a quadratic module $(M, Q)$, one can use the Eilenberg--MacLane theorem \cite{EM471, EM472} on abelian 3-cocycles to construct a unique (up to equivalence) pointed modular category $\cC(M, Q)$ \cite{JS85, JS93} (see also \cite[Thm.~8.4.9]{EGNO}).

More relevantly, each quadratic module $(M, Q)$ has an associated projective representation of $\SL$, which can be described as follows. The space of complex-valued functions on $M$, denoted by $V \defeq \BC^M$, is equipped with a natural Hermitian form 
\begin{equation*}
    \langle f, g \rangle \defeq \sum_{a \in M} f(a) \ol{g(a)}\,,
\end{equation*}
and we denote the vector norm of $f\in V$ by $\norm{f} \defeq \sqrt{\langle f, f \rangle}$. Note that $V$ admits a standard orthonormal basis: $\{\d_{a} \mid a \in M\}$. As described in \cite[Satz 2 \& Sec.~2]{Nobs1}, we have a projective representation
\[
    W(M, Q): \qsl{p^\ld} \to \PGL(V)
\]
defined by
\begin{equation}\label{eq:weil-rep}
\begin{split}
\fs\, \d_a \defeq W(M, Q)(\fs)(\d_a) &= \frac{\gamma_Q}{\abs{M}} \sum_{b \in M} \be(B(a,b))\,  \d_b\,, \\
\ft\,  \d_a \defeq W(M,Q)(\ft)(\d_a) &= \be(Q(a))\,  \d_a\,.
\end{split}
\end{equation}
Here $\gamma_Q \defeq \sum_{a \in M} \be(Q(a))$ is the Gauss sum of $(M,Q)$. This representation is called the \emph{Weil representation} associated to $(M, Q)$.  In fact, $W(M,Q)$ is precisely the projective representation $\tilde{\rho}_{\cC(M, Q)}$ arising from the pointed modular category $\cC(M, Q)$, as described in Example \ref{ex:mtc}; the modular data $(S,T)$ of $\cC(M, Q)$ is given by
\[
    S_{a,b} = \frac{1}{\sqrt{|M|}}\sum_{b \in M} \be(-B(a,b))\qquad \text{and}\qquad T_{a,b} = \be(Q(a)) \cdot \delta_{a,b} 
\]
for $a,b \in M$.  As noted in that example, $W(M,Q)$ can be rescaled to a linear representation of $\SL$, and the result is congruence and symmetric.

\subsection{Symmetrizability criteria}\label{subsec:criteria}

While it is immediate from \eqref{eq:weil-rep} that, for any quadratic module $(M,Q)$, the associated representation $W(M, Q)$ is symmetric, this does not necessarily imply that a given subrepresentation of $W(M, Q)$ is symmetrizable (as demonstrated in Example \ref{ex:nonsymmsub}).  To establish criteria for the symmetrizability of such subrepresentations, we use the following.

For any quadratic module $(M,Q)$, let $\Aut(M,Q)$ denote the group of automorphisms $\omega$ of the abelian group $M$ satisfying $Q(\omega a) = Q(a)$ for all $a \in M$.
For any $\omega \in \Aut(M,Q)$, we define the associated $\BC$-linear map $\varphi_{\omega}: V \to V$ by $\varphi_\omega(\delta_a) \defeq \delta_{\omega a}$ and the antilinear map $\ol{\varphi}_\omega$ as the composition of $\varphi_\omega$ and complex conjugation, relative to the standard basis $\{\delta_a \mid a \in M\}$ for $V=\BC^M$. Note that $\varphi_\omega$ preserves $\angleb{\cdot, \cdot}$, hence is an isometry on $V$ in the usual sense.

\begin{prop}\label{prop:kappa}
    Let $\omega \in \Aut(M,Q)$ be an involution and $\rho$ a subrepresentation of $W(M, Q)$ on $Y \subseteq V$. If $Y$ admits an orthonormal basis $\cB$ for which
    \begin{itemize}
    \item[(i)] $\cB$ is a set of eigenvectors of $\rho(\ft)$ and
    \item[(ii)] for any $f \in \cB$ such that $f$ and $\ol{\varphi}_\omega(f)$ are linearly independent, $\ol{\varphi}_\omega(f) \in \cB$,
    \end{itemize}
    then $\rho$ is symmetrizable.
\end{prop}

\begin{proof}
Let $\cB_1 \defeq \{f \in \cB \mid  f \text{ and } \ol{\varphi}_\omega(f) \text{ are linearly dependent}\}$. This means that, for each $f\in\cB_1$, there exists some $\eta_f \in \U(1)$ with $\ol{\varphi}_\omega(f) = \eta_f f$. Since $\omega^2 = \id$, $\ol{\varphi}_\omega^2 = \id$. So, we can choose $\cB_2\subset\cB\smallsetminus\cB_1$ such that $\cB_2\cap \ol{\varphi}_\omega(\cB_2) = \varnothing$ and $\cB = \cB_1 \sqcup \cB_2 \sqcup \ol{\varphi}_\omega(\cB_2)$. It is then clear that the set
\[
    \cS \defeq \{ \sqrt{\eta_f}f \mid f \in \cB_1 \} \sqcup \big\{ \frac{1}{\sqrt{2}}(f + \ol{\varphi}_{\omega}(f))~\big\vert~f \in \cB_2 \big\} \sqcup \big\{ \frac{i}{\sqrt{2}}(f - \ol{\varphi}_{\omega}(f))~\big\vert~f \in \cB_2 \big\}
\]
is an orthonormal basis for $Y$. Since $\ol{\sqrt{\e}} = \e \sqrt{\e}$ for $\e \in \{\pm 1\}$, we can also write $\cS$ as
\[
\cS = \{ \sqrt{\eta_f}f \mid f \in \cB_1 \} \sqcup \big\{ \frac{1}{\sqrt{2}}(\sqrt{\e}f + \ol{\sqrt{\e}}\ol{\varphi}_{\omega}(f))~\big\vert~\e\in\{\pm 1\}\,,\ f \in \cB_2 \big\}\,.
\]
It follows from the antilinearity of $\ol{\varphi}_\omega$ that $\ol{\varphi}_\omega(h) = h$ for all $h \in \cS$.

Finally, for each $f \in \cB$, we have $\rho(\ft)(f) = \xi_f f$ for some $\xi_f \in \U(1)$. Then
\[
    \rho(\ft)\ol{\varphi}_\omega(f) = \ol{\varphi}_\omega \rho(\ft)^{-1}(f) = \ol{\varphi}_\omega(\xi_f^{-1}f) = \xi_f\ol{\varphi}_\omega(f)\,.
\]
Therefore, $\cS$ is an eigenbasis for $\rho(\ft)$. By Lemma \ref{lem:sym-1}, $\cS$ is a symmetric basis for $\rho$, which means that $\rho$ is symmetrizable.
\end{proof}

\section{Irreducible representations of \texorpdfstring{$\qsl{p^\ld}$}{}} \label{sec:irreps}

In this section, we describe all of the irreducible representations of $\qsl{p^\ld}$ as per \cite{NW76}, where they are constructed using specific quadratic modules and their Weil representations. We show that all of these irreducible representations admit symmetries that enable us to apply the symmetrizability criteria established in Section \ref{sec:symm}. Finally, we complete the proof of Proposition \ref{prop:main} near the end of this section.

\subsection{Weil representations of prime power level}\label{subsec:nwreps}
Let $p$ be a prime and $\ld$ a positive integer. We follow \cite{Nobs1, NW76} and denote the ring $\BZ/p^{\ld}\BZ$ by $A_\ld$. By abuse of notation, we use $\fs$ and $\ft$ to denote both the generators of $\SL$ and their images in $\qsl{p^\ld}$. Clearly, any representation of $\qsl{p^\ld}$ is determined by the images of $\fs$ and $\ft$.

To construct irreducible representations of $\qsl{p^\ld}$, we consider the types of quadratic modules $(M,Q)$ described in Table \ref{tbl:weil_types}, wherein $M$ is an $A_\ld$-module (see~\cite[Def.~3]{Nobs1}).

\begin{table}[ht]
$\def\arraystretch{1.1}
\begin{array}{|l|c|c|c|c|c|c|}  
\hline
\text{Type}& p^\ld &  M & Q & \multicolumn{1}{c|}{\text{Other parameters}} & \fA \\
\hline\hline
\multirow{2}{*}{$D_{p^\ld}$} & \multirow{2}{*}{$\ld \ge 1$}  & \multirow{2}{*}{$A_\ld \oplus A_\ld$} & \multirow{2}{*}{$\dfrac{xy}{p^{\ld}}$} & 
&
\multirow{2}{*}{$A_{\ld}^{\times}$}\\
&&&&& \\
\hline
\multirow{4}{*}{$N_{p^\ld}$} & p=2 & \multirow{2}{*}{$A_\ld \oplus A_\ld$} & \multirow{2}{*}{$\dfrac{x^2 + xy + y^2}{2^{\ld}}$} &
& 
\multirow{8}{*}{$\{\e \in M^\times \mid \e \ol{\e} = 1\}$}
\\
&\ld \ge 1 &&&& \\
\cline{2-5}
& p \text{ odd} & \multirow{2}{*}{$A_\ld \oplus A_\ld$} & \multirow{2}{*}{$\dfrac{x^2 + xy + \frac{1+t}{4}y^2}{p^{\ld}}$} & t \in \BN, \,\big(\frac{-t}{p}\big) = -1
& 
\\
&\ld \ge 1 &&& t \equiv 3 \mod 4 & \\
\cline{1-5}
\multirow{4}{*}{$R_{p^\ld}^\s(r,t)$} & p=2 & \multirow{2}{*}{$A_{\ld-1} \oplus A_{\ld-\s-1}$} & \multirow{2}{*}{$\dfrac{r(x^2 + 2^\s t y^2)}{2^{\ld}}$} & 0 \leq \s \leq \ld-2 &
\\
& \ld \geq 2 &&& r,t \in \BN \text{ and odd} & \\
\cline{2-5}
& p \text{ odd} & \multirow{2}{*}{$A_{\ld} \oplus A_{\ld-\s}$} & \multirow{2}{*}{$\dfrac{r(x^2 + p^\s t y^2)}{p^{\ld}}$} & 1 \leq \s \leq \ld-1 & \\
& \ld \geq 2 &&& r,t \in \{1,u\}& \\
\hline
\multirow{2}{*}{$R_{p^\ld}(r)$} & p \text{ odd} & \multirow{2}{*}{$A_{\ld}$} & \multirow{2}{*}{$\dfrac{rx^2}{p^{\ld}}$} & \multirow{2}{*}{$r \in \{1,u\}$}&
\\
& \ld \ge 1 &&&&\\
\hline
\end{array}$
\caption{Types of quadratic modules with at most two elementary divisors.}
\label{tbl:weil_types}
\end{table}
Here $u$ is a fixed quadratic nonresidue mod $p$.  The group $\fA$ will be explained in Section \ref{subsec:std-rep}.

Each choice of $M$ has a ring structure. Types $D_{p^\ld}$ and $R_{p^\ld}(r)$ are equipped with their natural ring structure. For the others, we may identify $M$ with a quotient ring as follows:
\begin{itemize}
    \item for type $N_{2^\ld}$, let $X \defeq \frac{1}{2}(1 + \sqrt{-3})$, and then $M \defeq A_\ld \oplus A_\ld \cong \BZ[X] / (2^\ld)$\,,
    \item for type $N_{p^\ld}$ with $p$ odd, let $X \defeq \frac{1}{2}(1 + \sqrt{-t})$, and then $M \defeq A_\ld \oplus A_\ld \cong \BZ[X] / (p^\ld)$\,,
    \item for type $R_{2^\ld}^\s(r,t)$, let $X \defeq \sqrt{-2^\s t}$, and then $M \defeq A_{\ld-1} \oplus A_{\ld-\s-1} \cong \BZ[X] / (2^{\ld-\s-1}X)$\,,
    \item for type $R_{p^\ld}^\s(r,t)$ with $p$ odd, let $X \defeq \sqrt{-p^\s t}$, and then $M \defeq A_{\ld} \oplus A_{\ld-\s} \cong \BZ[X] / (p^{\ld-\s}X)$\,.
\end{itemize}
In each case, we identify $(x,y)$ with $x + Xy$.  The $A_\ld$-module $M$ then inherits the multiplication and complex conjugation of the quotient ring as well as the norm of $\BZ[X]$.  In particular, for $N_{p^\ld}$, $Q(x,y) = \operatorname{Norm}(x,y)/p^\ld$; while for $R_{p^\ld}^\s(r,t)$, we have $Q(x,y) = r \cdot \operatorname{Norm}(x,y)/p^\ld$.  We write $M^\times$ for the multiplicative group of units of $M$.

For each of these types, the projective Weil representation $W(M,Q)$ defined by \eqref{eq:weil-rep} is in fact a linear representation of $\qsl{p^\ld}$ \cite[Sec.~2]{Nobs1}.

\subsection{Standard irreducible representations}
\label{subsec:std-rep}
The quadratic modules of type $D_{p^\ld}$, $N_{p^\ld}$, and $R_{p^\ld}^\s(r,t)$, as described in Table \ref{tbl:weil_types}, will simply be referred as \emph{binary quadratic modules} throughout this paper, as $M$ has exactly 2 elementary divisors. 

For any binary quadratic module $(M, Q)$, we define $\k \in \Aut(M,Q)$ as follows: 
\[
    \k \defeq 
    \begin{cases}
    (x, y) \mapsto (y, x)\,, & \text{if } (M,Q) \text{ is of type } D_{p^\ld}\,;\\
    (x, y) \mapsto \overline{(x, y)} = (x+y,-y)\,, & \text{if } (M,Q) \text{ is of type } N_{p^\ld}\,;\\
    (x, y) \mapsto \overline{(x, y)} = (x,-y)\,, & \text{if } (M,Q) \text{ is of type } R_{p^\ld}^\s(r,t)\,.
    \end{cases}
\]
From the definition of $Q$ in Table \ref{tbl:weil_types}, it is immediate that $\k \in \Aut(M, Q)$.  Note that $\k$ is of order 2, except in the case of $R_{2^\ld}^{\ld-2}(r,t)$, where the second factor of $M$ is isomorphic to $\BZ/2\BZ$ and hence $\k$ is trivial.
\begin{defn}
    A binary quadratic module of type $R_{2^\ld}^{\ld-2}(r,t)$ is called \emph{extremal}.
\end{defn}

Next, we associate to each binary quadratic module $(M,Q)$ a particular abelian subgroup $\fA \leq \Aut(M,Q)$ as follows. If $(M,Q)$ is of type $D_{p^\ld}$, the group $\fA \defeq A_{\ld}^\times$ acts on $M$ via $\e(x,y) = (\e\inv x, \e y)$ for any $\e\in\fA$ and $(x,y) \in M$; if $(M,Q)$ is of type $N_{p^\ld}$ or $R_{p^\ld}^\s(r,t)$, we take $\fA \defeq \{\e \in M^\times \mid \e \ol{\e} = 1\}$, acting on $M$ by multiplication (see Section \ref{subsec:nwreps}). In each case, we can check that $\fA$ is indeed an abelian subgroup of $\Aut(M,Q)$. Note that, in the case of an extremal quadratic module $(M, Q)$, we have $\ol{a} = a$ for all $a \in M$, so $\fA = \{\e \in M^\times \mid \e^2 = 1\}$ is an elementary 2-group. We also have the following lemma.

\begin{lemma}\label{lem:k-a} 
Let $(M,Q)$ be a binary quadratic module. For any $\e \in \fA$, $(\k \circ \e)^2  = \id$.
\end{lemma}
\begin{proof}
Indeed, for type $D_{p^\ld}$, we have
\[(\k \circ \e)^2  (x,y) = \k \big( \e (\e y, \e^{-1} x)\big) = (x,y)\]
for all $(x,y) \in M$. For type $N_{p^\ld}$ or $R_{p^\ld}^\s(r,t)$, we have $\ol{\e}= \e^{-1}$ and thus \[(\k \circ \e)^2(a) = \k \big( \e (\ol{\e}\,\ol{ a })\big) = a\]
for all $a \in M$. As a particular case, for the extremal type $R_{2^\ld}^{\ld-2}(r,t)$, $\fA$ has exponent $2$ and $\k = \id$, so the condition $(\k \circ \e)^2 = \id$ follows trivially.
\end{proof}

Characters of $\fA$ naturally give rise to subrepresentations of $W(M,Q)$. More precisely, denote by $\hat{\fA}$ the character group of $\fA$. Then, for any $\chi \in \hat{\fA}$,
\begin{equation}\label{eq:V-chi}
    V^{\chi} \defeq \{ f \in \BC^M \mid f(\e a) = \chi(\e)f(a) \text{ for all } a \in M \text{ and } \e \in \fA \}
\end{equation}
is an $\qsl{p^\ld}$-invariant subspace of $V$. The restriction of $W(M,Q)$ to $V^\chi$ is denoted by $W(M,Q,\chi)$. Using \eqref{eq:V-chi} and Lemma \ref{lem:k-a}, it is straightforward to verify that $\varphi_\k$ (as defined in Section \ref{subsec:criteria}) maps $V^\chi$ to $V^{\ol{\chi}}$. In fact, $W(M, Q, \chi)$ is equivalent to $W(M, Q, \ol{\chi})$ via $\varphi_\k$.

A basis for $V^\chi$ can be chosen as follows (cf.~\cite{NW76}). For any $\chi \in \hat{\fA}$ and $a \in M$, define
\begin{equation*}
    \tilde{f}_{a}^\chi \defeq 
    \sum_{\e\in\fA} \chi(\e) \d_{\e a}\,.
\end{equation*}
Clearly, we have $\tilde{f}^\chi_a \in V^\chi$. Whenever $\tilde{f}_{a}^{\chi}\ne 0$ (which occurs if and only if $\Stab(a)\subseteq \ker(\chi)$), define
\begin{equation*}
    f_a^\chi \defeq \frac{\tilde{f}_{a}^\chi}{\norm{\tilde{f}_{a}^\chi}}~.
\end{equation*}
Let $\theta$ be a complete set of representatives for the orbits of $\fA$ on $M$ such that, for any $a \in \theta$, if $\k a \notin \fA a$, then $\k a\in \theta$. Define 
\[
    \theta^\chi \defeq \theta \cap \{a \in M \mid \Stab(a) \subseteq \ker(\chi)\}\,.
\] 
By Lemma \ref{lem:k-a}, $\Stab(\k a) = \Stab(a)$ for any $a \in M$, so the assumption on $\theta$ ensures that, if $a \in \theta^\chi$ and $\k a \notin \fA a$, then $\k a \in \theta^\chi$. Moreover, since the $\fA$-orbits are disjoint, the set
\begin{equation*}
\label{eq:b-chi}
    \cB^\chi \defeq \left\{
    f_{a}^\chi~\middle\vert~a\in \theta^\chi\right\}
\end{equation*}
is an orthonormal basis for $V^\chi$.

\begin{prop}\label{prop:V-chi}
Let $(M, Q)$ be a binary quadratic module. Then, for any character $\chi \in \hat{\fA}$, $W(M,Q,\chi)$ is symmetrizable.
\end{prop}
\begin{proof}
It suffices to show that the basis $\cB^\chi$ defined above satisfies the conditions in Proposition \ref{prop:kappa}.

Recall that for, any $a \in \theta^\chi$ and $\e\in\fA$, we have $Q(\e a) = Q(a)$. As such, \eqref{eq:weil-rep} yields
\begin{equation}\label{eq:tfachi}
\ft f^\chi_a = \frac{1}{\norm{\tilde{f}_{a}^{\chi}}}\sum_{\e\in\fA} \chi(\e) \ft\d_{\e a} = \frac{1}{\norm{\tilde{f}_{a}^{\chi}}}\sum_{\e\in\fA} \chi(\e) \be(Q(\e a))\d_{\e a} = \be(Q(a)) f_a^\chi\,.
\end{equation}
Thus, $\cB^\chi$ is an eigenbasis for $\ft$.

Further, by definition and Lemma \ref{lem:k-a}, for any $a \in \theta^\chi$, we have
\begin{equation}\label{eq:phi-k-f}
\ol{\varphi}_\k(f^\chi_a) = \frac{1}{\norm{\tilde{f}_{a}^{\chi}}}\sum_{\e\in\fA} \chi(\e^{-1}) \d_{\k\e a} = \frac{1}{\norm{\tilde{f}_{a}^{\chi}}}\sum_{\e\in\fA}\chi(\e^{-1}) \d_{\e^{-1}\k a} = f^\chi_{\k a}\,,
\end{equation}
noting that $\norm{\tilde{f}^\chi_a} = \norm{\tilde{f}^\chi_{\k a}}$. If $\k a \in \fA a$, then $\k a = \mu_a a$ for some $\mu_a \in \fA$.  This implies $f^{\chi}_{\k a} = f^{\chi}_{\mu_{a} a} = \chi(\mu_a^{-1})f^{\chi}_{a}$, and hence 
$f_a^\chi$ and $\ol\varphi_\k(f_a^\chi)$ are linearly dependent. Thus, if $f_a^\chi$ and $\ol\varphi_\k(f_a^\chi)$ are linearly independent, then  $\k a \notin \fA a$. By the assumption on $\theta$ and the preceding discussion, $\k a \in \theta^\chi$, and so $f^\chi_{\k a} = \ol\varphi_\k(f^\chi_a) \in \cB^\chi$. The result now follows from Proposition \ref{prop:kappa}.
\end{proof}

\begin{remark}\label{rmk:S-V-chi}
Let $\theta_1^\chi \defeq \{a \in \theta^\chi\mid \k a \in \fA a\}$. Then, by the proof of Proposition \ref{prop:kappa}, there is a choice of subset $\theta_2^\chi\subset\theta^\chi$ such that $\theta_2^\chi \cap \k(\theta_2^\chi) = \varnothing$ and $\theta^\chi = \theta^\chi_1 \cup \theta_2^\chi \cup \k(\theta_2^\chi)$. Moreover, a symmetric basis of $W(M, Q, \chi)$ can be chosen to be 
\[
    \cS^\chi = 
    \big\{ \sqrt{\chi(\mu_a^{-1})} f_a^\chi ~\big\vert~ a \in \theta_1^\chi \big\} \cup 
    \big\{ \frac{1}{\sqrt{2}}(f_a^\chi + f_{\k a}^\chi) ~\big\vert~ a \in \theta_2^\chi \big\} \cup
    \big\{ \frac{i}{\sqrt{2}}(f_a^\chi - f_{\k a}^\chi) ~\big\vert~ a \in \theta_2^\chi \big\}\,,
\]
where the notation $\mu_a$ is as in the proof of Proposition \ref{prop:V-chi}.
\end{remark}

When $\chi^2 = 1$ (i.e.~$\chi = \ol\chi$), $\varphi_\k$ becomes an auto-equivalence of $V^\chi$. Therefore, in this case, if $\varphi_\k\restr{V^{\chi}}\ne\id$, then $W(M, Q, \chi)$ admits a further decomposition into eigenspaces of $\varphi_\k$:
\[
    V^\chi_\pm \defeq \{ f \in V^\chi \mid f(\k a) = \pm f(a) \text{ for all } a \in M \}\,.
\]
The corresponding subrepresentations are denoted by $W(M, Q, \chi)_{\pm}$.

\begin{prop}\label{prop:V-chi-pm}
Let $(M, Q)$ be a binary quadratic module. Then, for any $\chi \in \hat{\fA}$ satisfying $\chi^2 = 1$ and $\varphi_\k\restr{V^{\chi}} \ne \id$, the subrepresentations $W(M, Q, \chi)_{\pm}$ are both symmetrizable.
\end{prop}
\begin{proof}
It suffices to show that every element in the symmetric basis $\cS^\chi$ for $V^\chi$ in Remark \ref{rmk:S-V-chi} is an eigenvector of $\varphi_\k$, since this will imply that $\cS^\chi_{\pm} \defeq V^\chi_{\pm} \cap \cS^\chi$ are symmetric bases for $W(M, Q, \chi)_{\pm}$.

By \eqref{eq:phi-k-f}, for any $a \in \theta^\chi$, we have $\varphi_\k(f_a^\chi) = \ol{f^{\chi}_{\k a}}$. Moreover, since $\chi^2 = 1$, we have $\ol{f^{\chi}_{\k a}} = f^\chi_{\k a}$, which means $\varphi_\k(f^\chi_a) = f^\chi_{\k a}$. Therefore, for any $a \in \theta_2^\chi$, it is readily seen that $\frac{1}{\sqrt{2}}(f_a^\chi + f_{\k a}^\chi) \in V_+^\chi$, and $\frac{i}{\sqrt{2}}(f_a^\chi - f_{\k a}^\chi) \in V_-^\chi$.
 
Finally, for any $a \in \theta^\chi_1$, $\k a = \mu_a a$ for some $\mu_a \in \fA$. In this case, the same computation as in the proof of Proposition \ref{prop:V-chi} shows that $\varphi_\k(f_a^\chi) = f^\chi_{\k a} = \chi(\mu_a^{-1}) f^{\chi}_{a}$, which equals $\pm f^{\chi}_{a}$ as $\chi^2 = 1$. This completes the proof.
\end{proof}

The question of which characters $\chi \in \hat{\fA}$ give rise to irreducible $W(M, Q, \chi)$ was answered as a remarkable result of \cite{NW76}; we need the following definition for the statement.

\begin{defn}
    Let $(M,Q)$ be a binary quadratic module which is not extremal, and let $\fA \leq \Aut(M,Q)$ be the corresponding abelian subgroup. A character $\chi \in \hat{\fA}$ is called \emph{primitive} if there exists some $\e \in \fA$ such that $\chi(\e) \neq 1$ and $\e$ fixes $pM$ pointwise.
\end{defn}

Nobs and Wolfart showed that most primitive characters of $\fA$ give rise to irreducible representations. More precisely, they proved the following theorem.

\begin{thm}[{\cite[Hauptsatz 1]{NW76}}]\label{thm:nw1}
    Let $(M,Q)$ be a quadratic module of type $D_{p^\ld}$, $N_{p^\ld}$, or non-extremal $R_{p^\ld}^\s(r,t)$, and let $\fA \leq \Aut(M,Q)$ be the corresponding subgroup.
    If $\chi \in \hat{\fA}$ is primitive and not an involution, then $W(M,Q,\chi)$ is an irreducible representation of $\SL$ of level $p^\ld$.
    
    If $\chi_1$,  $\chi_2 \in \hat{\fA}$ are primitive and not involutions, then $W(M,Q,\chi_1)$ is equivalent to $W(M,Q,\chi_2)$ if, and only if, $\chi_1 = \chi_2$ or $\chi_1 = \ol{\chi_2}$. 
\end{thm}

The case of $\chi^2 = 1$ is not directly covered by the theorem, but $W(M, Q, \chi)_{\pm}$ is irreducible in many cases. The precise details can be found in the complete list of irreducible representations of $\qsl{p^\ld}$ in \cite[pp.~521-525]{NW76}.

\begin{defn}
Let $p$ be a prime and $\ld \in \BN$. We will call an irreducible representation of $\qsl{p^\ld}$ that is equivalent to $W(M, Q, \chi)$ or $W(M, Q, \chi)_\pm$ for some binary quadratic module $(M,Q)$ a \emph{standard irreducible representation}. 
\end{defn}

Combining Propositions \ref{prop:V-chi} and \ref{prop:V-chi-pm}, we have:
\begin{prop}\label{prop:std}
For any prime $p$ and positive integer $\ld$, every standard irreducible representation of $\qsl{p^\ld}$ is symmetrizable.\qed
\end{prop}

\subsection{Special irreducible representations of \texorpdfstring{$\qsl{2^\ld}$}{}}\label{subsec:spc-rep}

For a quadratic module $(M,Q)$ of type $R_{2^\ld}^\sigma(r,t)$ and $\chi \in \hat \fA$, we denote the representation $W(M,Q,\chi)$ of $\qsl{2^\ld}$ by $R_{2^\ld}^\sigma(r,t, \chi)$.  A representation of the form $R_{2^\ld}^\sigma(r, t, \chi)$ with $\chi$ not primitive is usually reducible, but some cases with $\sigma = \ld-2$ or $\ld - 3$ will contain a unique irreducible subrepresentation of level $2^\ld$ that does not occur among the standard representations \cite[Sec.~6]{NW76}. We will call the irreducible representations appearing this way \emph{special}; they are denoted by $R_{2^\ld}^\sigma(r,t, \chi)_1$. We list all the special irreducible representations (up to equivalence), together with a choice of basis for each, in Table \ref{tbl:special}.

\begin{table}[ht]
\centering
$\def\arraystretch{1.2}
\begin{array}{|c|c|l|}
\hline
\text{Type} & M & \multicolumn{1}{|c|}{\text{Basis in \cite{NW76}}} \\
\hline\hline
R_{2^2}^0(1,3,\chi_1)_1 & A_1 \oplus A_1 & 
\d_{(1,0)},\, \d_{(0,1)},\,  \d_{(0,0)}-\d_{(1,1)} \\
\hline 
\multirow{2}{*}{$R_{2^3}^0(1,3,\chi_1)_1$} & 
\multirow{2}{*}{$A_2 \oplus A_2$} &
\d_{(0,0)}-\d_{(2,2)},\, \d_{(2,0)}-\d_{(0,2)},\,  \d_{(1,0)}+\d_{(-1,0)}, \\
&&\d_{(1,2)}-\d_{(-1,2)},\,  \d_{(0,1)}+\d_{(0,-1)},\,  \d_{(2,1)}+\d_{(2,-1)} \\
\hline 
R_{2^4}^2(r,3,\chi_1)_1 & \multirow{2}{*}{$A_3 \oplus A_1$} &
\d_{(1,0)}+\d_{(-1,0)},\,  \d_{(3,0)}+\d_{(-3,0)},\,  \d_{(1,1)}+\d_{(-1,1)}, \\
r\in\{1,3\} & & 
\d_{(3,1)}+\d_{(-3,1)},\, \d_{(0,0)}-\d_{(4,0)},\, \d_{(0,1)}-\d_{(4,1)} \\
\hline 
R_{2^5}^{2}(r,1,\chi_1)_1 & \multirow{2}{*}{$A_4 \oplus A_2$} & 
\tilde{f}_{a}^{\chi_{1}} \text{ for } a \in \{1,3,5,7\} \times \{0,1\},\,
\tf_{(2,0)}^{\chi_{1}}-\tf_{(6,0)}^{\chi_{1}}, \\
r \in \{1,3\} & &  \tf_{(2,2)}^{\chi_1}-\tf_{(6,2)}^{\chi_1},\, \tf_{(0,0)}^{\chi_1}-\tf_{(8,0)}^{\chi_1},\, \tf_{(0,2)}^{\chi_1}-\tf_{(8,2)}^{\chi_1} \\
\hline 
R_{2^5}^{2}(r,1,\chi_2)_1 & \multirow{2}{*}{$A_4 \oplus A_2$} & 
\tilde{f}_{a}^{\chi_2} \text{ for } a \in \{1,3,5,7\} \times \{0,1\}, \\
r \in \{1,3\} & & \tf_{(4,0)}^{\chi_2}, \tf_{(4,2)}^{\chi_2},\, \tf_{(2,0)}^{\chi_2}-\tf_{(6,0)}^{\chi_2},\, \tf_{(2,2)}^{\chi_2}-\tf_{(6,2)}^{\chi_2} \\
\hline 
R_{2^{6}}^4(r,t,\chi_1)_1 & \multirow{2}{*}{$A_5 \oplus A_1$} & \tf_{(x,0)}^{\chi_1} \text{ for odd } 1 \le x \le 15,\,
\tf_{(0,0)}^{\chi_1}-\tf_{(16,0)}^{\chi_1}, \\
(r,t) \in \{1,3,5,7\} \times \{1,3\} & &  \tf_{(4,0)}^{\chi_1}-\tf_{(12,0)}^{\chi_1},\, \tf_{(2,1)}^{\chi_1}-\tf_{(14,1)}^{\chi_1},\, \tf_{(6,1)}^{\chi_1} -\tf_{(10,1)}^{\chi_1} \\
\hline 
R_{2^{\ld}}^{\ld-3}(r,t,\chi)_1 & \multirow{3}{*}{$A_{\ld-1}\oplus A_2$} &
\text{See table at \cite[p.~512]{NW76}. The basis elements are of}\\
(r,t) \in \{1,3,5,7\} \times \{1,3\},  & & 
\text{the form } \tf_{a}^{\chi} \text{ for some } a \in Y_0
\text{, or } \tf_{(x,y)}^\chi - \tf_{(2^{\ld-2}-x,y)}^\chi\\
\ld \ge 7, \chi \in \angleb{\chi_3} & & 
\text{for some } (x,y) \in Y_1\,.\\
\hline
\end{array}$
\caption{Special irreducible representations.}
\label{tbl:special}
\end{table}

In this table we use the following notation. Let $\chi_1$ denote the trivial character. For $R_{2^5}^{2}(r,1)$, we have $\fA = \angleb{(-1,0)} \times \angleb{(9,2)}$, and $\chi_2$ denotes the character determined by $\ker(\chi_2) = \angleb{(9,2)}$. Finally, for $R_{2^\ld}^{\ld - 3}(r,t)$ with $\ld \ge 7$, we have $\fA = \angleb{(-1,0)}\times \angleb{\alpha}$, where $\alpha = (1-2^{\ld-4}t-2^{2\ld-9},\, 1)$, and $\chi_3$ denotes the character determined by $\ker(\chi_3) = \angleb{(-1,0)}$. The sets $Y_0$ and $Y_1$ are defined as the following disjoint unions:
\[\begin{split}
    Y_0 &\defeq \{(x, 0) \mid x \text{ odd}\} \sqcup \{(x, y) \mid y \in \{0, 2\},\, x = 4-2y + 8j,\, 0 \le j \le 2^{\ld - 6}-1\}\,, \\
    Y_1 &\defeq \{(x, y) \mid y \in \{0, 2\},\, x = 2y+8j,\, 0 \le j \le 2^{\ld - 6} - 1\} \sqcup \{(x, 0) \mid x = 2 + 4k,\, 0 \le k \le 2^{\ld - 5} - 1\}\,.
\end{split}\]

We may then derive the following proposition.

\begin{prop}\label{prop:spc}
    Every special irreducible representation in Table \ref{tbl:special} is symmetrizable.
\end{prop}
\begin{proof}
We will apply Lemma \ref{lem:sym-1} to show that each basis in the table is a symmetric basis for the corresponding representation.

First, we observe that each basis in the table is an orthogonal basis.  Indeed, this is clear for the first six rows. For the last row, it follows from the fact that $Y_0$ and $Y_1$ are disjoint.

Next, we claim that each basis element in the table is fixed by $\ol{\varphi}_\k$, as follows.  Recall that $\k(x,y) = (x,-y)$ for type $R^\s_{2^\ld}(r,t)$.  It is immediate from this that each basis element in the first three rows is fixed by $\ol\varphi_\k$.

For $R_{2^5}^2(r,1)$, direct computation yields $(9,2) \cdot (x,1) = (x,-1)$ for each $x \in \{1,3,5,7\}$. If $\chi = \chi_1$ or $\chi_2$, then $\chi^2 = 1$, so 
\[\ol\varphi_\k(\tf_{(x,1)}^\chi) = \varphi_\k(\tf_{(x,1)}^\chi) = \tf_{(x,-1)}^\chi = \chi(9,2)\tf_{(x,1)}^\chi = \tf_{(x,1)}^\chi\] 
for any $x \in \{1,3,5,7\}$. Moreover, since $M = A_4 \oplus A_2$, for any $(x, y) \in A_4 \times \{0,2\}$, we have \[\ol\varphi_k(\tf_{(x, y)}^{\chi}) = \varphi_k(\tf_{(x, y)}^{\chi}) = \tf_{(x, -y)}^{\chi} = \tf_{(x, y)}^{\chi}\,.\] This confirms that each basis element in the $4\nth$ and $5\nth$ rows is fixed by $\ol\varphi_\k$.

For $R_{2^6}^4(r,t)$, $M = A_5\oplus A_1$, so $\k$ acts trivially on $M$. Hence, for any $a \in M$, the function $\tf_a^{\chi_1}$ is fixed by $\ol\varphi_\k$. Since $\ol\varphi_\k$ is antilinear, it also fixes the other basis elements, as each is a $\BZ$-linear combination of $\tf_a^{\chi_1}$.

Similarly, for $R_{2^\ld}^{\ld - 3}(r,t)$ with $\ld \ge 7$, we have $M = A_{\ld - 1} \oplus A_2$, so (again) $\k(x,y) = (x,y)$ for any $(x, y) \in A_{\ld-1}\times \{0,2\}$. Therefore, for any $(x, y) \in A_{\ld-1}\times \{0,2\}$, the function $\tf_{(x,y)}^\chi$ is fixed by $\ol\varphi_\k$. Since $\ol\varphi_\k$ is antilinear, it also fixes the rest of the basis elements.

Finally, we claim that each basis element in the table is an eigenvector for $\ft$. Indeed, for any quadratic module $(M, Q)$ of type $R_{2^\ld}^\sigma(r, t)$, \eqref{eq:weil-rep} and \eqref{eq:tfachi} show that any function of the form $\d_a$ or $\tf_{a}^\chi$ for $a \in M$ and $\chi \in \fA$ is an eigenvector of $\ft$ with eigenvalue $\be(Q(a))$. To show a basis element in Table \ref{tbl:special} is an eigenvector of $\ft$, it suffices to show that the value of $Q(a)$ is the same for each index $a \in M$ among its summands.  Recall that $Q(x,y) = r(x^2 + 2^\sigma ty^2)/2^\ld \in \BQ/\BZ$ in this case. In particular, for $(x,y)\in M$, we have $Q(x, y) = Q(-x,y) = Q(x,-y)$. Our claim then follows from the computations below.
\begin{itemize}
\item 
For $R_{2^2}^0(1,3)$, $Q(0,0) = Q(1,1) = 0$.
\item 
For $R_{2^3}^0(1,3)$, $Q(0,0) = Q(2,2) = 0$ and $Q(2,0) = Q(0,2) = 1/2$.
\item
For $R_{2^4}^2(r,3)$, $Q(0,0) = Q(4,0) = 0$ and $Q(0,1) = Q(4,1) = 3r/4$.
\item
For $R_{2^5}^2(r,1)$, $Q(2,0) = Q(6,0) = r/8$, $Q(2,2) = Q(6,2) = 5r/8$, $Q(0,0) = Q(8,0) = 0$, and $Q(0,2) = Q(8,2) = r/2$.
\item
For $R_{2^6}^4(r,t,\chi)_1$, the basis elements are either of the form $\tf^{\chi_{1}}_{a}$ for some $a\in M$, or of the form $\tf^{\chi_1}_{(2x,y)} - \tf^{\chi_1}_{(16-2x,y)}$ for some $(2x,y) \in M$. As such, it suffices to verify the following equality for any $(2x,y)\in M$:
\[Q(16-2x,y) = \frac{r((16-2x)^2 + 16ty^2)}{64} =  \frac{r(4x^2+16ty^2)}{64} =Q(2x,y)\,.\]
\item
For $R_{2^\ld}^{\ld-3}(r,t)$ with $\ld \ge 7$, any element in $Y_1$ is of the form $(2u, v) \in A_{\ld-1} \times \{0, 2\}$ by definition. Now, we find
\[
Q(2^{\ld-2}-2u, v) = \frac{r((2^{\ld-2}-2u)^2 + 2^{\ld-3}tv^2)}{2^\ld} = \frac{r(2^{2\ld-4} - 2^\ld u + 4u^2 + 2^{\ld-3}tv^2)}{2^\ld} = Q(2u,v)\,.
\]
\end{itemize}

In summary, each of the bases in Table \ref{tbl:special} is an orthogonal eigenbasis for $\ft$, and each basis element thereof is fixed by $\ol\varphi_\k$. Therefore, the normalization of these bases are symmetric bases for the corresponding representations by Lemma \ref{lem:sym-1}, and this completes the proof. 
\end{proof}

\subsection{Unary representations}\label{sec:r_unary}

\emph{Unary quadratic modules} are those of type $R_{p^\ld}(r)$, where $p$ is an odd prime and $M = A_\ld$ is cyclic. In this case, it is easy to see $\Aut(M,Q) = \{\pm 1\}$, and we define
\[
    \k : M \to M\,,\qquad a \mapsto -a\,.
\]
The representation $W(M,Q)$, denoted simply by $R_{p^\ld}(r)$, decomposes into two subrepresentations $R_{p^\ld}(r)_\pm$ corresponding to the $(\pm1)$-eigenspaces of $\varphi_\k$. For $\ld = 1$, these are irreducible.  For $\ld \geq 2$, each contains a unique irreducible subrepresentation of level $p^\ld$, denoted $(R_{p^\ld}(r)_\pm)_1$. Specifically, \cite[Satz 8]{NW76} shows that
\[
    R_{p^\ld}(r) \cong (R_{p^\ld}(r)_+)_1 \oplus (R_{p^\ld}(r)_-)_1 \oplus R_{p^{\ld-2}}(r)
\]
(wherein $R_1(r)$ is the trivial  representation). We will call the irreducible representations $R_p(r)_\pm$ ($\ld = 1$) and $(R_{p^{\ld}}(r)_\pm)_1$ ($\ld \ge 2$) for any odd prime $p$ the \emph{unary irreducible representations} of $\qsl{p^\ld}$.

With some minor changes from \cite{NW76}\footnote{Cf.~\cite[p.~509]{NW76}. With $g_{y,k,\e}$ as defined in loc.~cit., here we have $h_{y,k,\e,\eta} = \frac{1}{2\sqrt{p}} ( g_{y,k,\e} + \e\eta g_{(p^{\ld-1}-y),k,\e} )$.}, an orthonormal basis for each unary irreducible representation can be chosen as follows. For $x \in M = A_\ld$ and $\e \in \{\pm 1\}$, define
\begin{equation*}
\tilde{f}_{x,\e} \defeq \sqrt{\e}\d_x + \ol{\sqrt{\e}} \d_{-x} = \sqrt{\e}\d_x + \ol\varphi_{\k}(\sqrt{\e}\d_x)
\qquad\text{and}\qquad
f_{x,\e} \defeq \frac{1}{\sqrt{2}}\tilde{f}_{x,\e}\,.
\end{equation*}
In particular, we have
\begin{equation}\label{eq:unary-k}
\ol\varphi_\k(\tf_{x, \e}) = \tf_{x, \e} \qquad\text{and}\qquad
\ol\varphi_\k(f_{x, \e}) = f_{x, \e}\,.
\end{equation}
Note also that, by \eqref{eq:weil-rep} and $Q(x) = Q(-x) = rx^2/p^{\ld}$, $\tf_{x,\e}$ and $f_{x, \e}$ are eigenvectors of $\ft$.

Further, for $0 \le y < p^{\ld-1}$, $1 \le k < p$, and $\e, \eta \in \{\pm 1\}$, define
\begin{equation*}\begin{split}
    h_{y,k,\e,\eta} &\defeq \frac{1}{\sqrt{p}}\sum_{a \in A_1} \left(\sqrt{\eta}\zeta_p^{ka} + \ol{\sqrt{\eta}\zeta_p^{ka}}\right) \tilde{f}_{(py+ap^{\ld-1}),\,\e}\,.
\end{split}\end{equation*}
By \eqref{eq:unary-k} and the antilinearity of $\ol\varphi_\k$, we have $h_{y,k,\e,\eta} = \ol\varphi_\k(h_{y,k,\e, \eta})$. Moreover, for any $\ld \ge 2$ and any integers $y$ and $a$, we have 
\[
    Q(py+ap^{\ld-1}) = \frac{r(py+ap^{\ld-1})^2}{p^{\ld}} = \frac{r((py)^2 + 2ap^\ld + a^2p^{2\ld-2})}{p^{\ld}} = \frac{r(py)^2}{p^\ld} = Q(py) \in \BQ/\BZ\,.
\]
Therefore, for $\ld \ge 2$, $h_{y,k,\e,\eta}$ is an eigenvector of $\ft$.
Then, denoting
\begin{equation*}\begin{split}
    \cF_\e &\defeq \left\{f_{x,\e} ~\middle\vert~ x \in M^\times \text{ with } 1 \leq x \leq \frac{p^\ld - 1}{2}\right\}\,, \\
    \cH_\e &\defeq \left\{h_{y,k,\e,\eta} ~\middle\vert~ 1 \leq y \leq \frac{p^{\ld-2}-1}{2},~1 \leq k \leq \frac{p-1}{2},~\eta \in \{\pm1\}\right\}\,,
\end{split}\end{equation*}
we have the following orthonormal eigenbases for $\ft$:
\begin{itemize}
    \item For $R_p(r)_+$,
    \(
        \cB \defeq \cF_{+1} \cup \{ \d_0\}\,.
    \)
    \item For $R_p(r)_-$,
    \(
        \cB \defeq \cF_{-1}\,.
    \)
    \item For $(R_{p^\ld}(r)_\e)_1$ with $\ld \geq 2$,
    \begin{equation*}\begin{split}
        \cB \defeq \cF_{\e} \cup \cH_{\e} \cup \left\{\frac{1}{\sqrt{2}}h_{0,k,\e,\e}~\middle|~1 \leq k \leq \frac{p-1}{2}\right\}\,.
    \end{split}\end{equation*}
\end{itemize}
By the above discussions, for each unary irreducible representation, the corresponding basis $\cB$ is an orthonormal $\ft$-eigenbasis that is fixed by $\ol\varphi_\k$ elementwise. Therefore, by Lemma \ref{lem:sym-1}, $\cB$ is a symmetric basis. In other words, we have the following proposition.

\begin{prop}\label{prop:unary}
Every unary irreducible representation is symmetrizable. \qed
\end{prop}

\subsection{Proof of Proposition \ref{prop:main} and applications}
We are now ready to prove Proposition \ref{prop:main}.
\begin{proof}[Proof of Proposition \ref{prop:main}]
According to \cite[Hauptsatz 2]{NW76} (see also the tables in \cite[pp.~521-525]{NW76}), every irreducible representation of $\qsl{p^\ld}$ is equivalent to one of the following: a standard irreducible representation, a special irreducible representation, a unary irreducible representation, or a tensor product of two representations of the above three types. Since symmetrizability is preserved under taking tensor product (see Remark \ref{rem:preservation}) and each of the first three types of representations is symmetrizable by Propositions \ref{prop:std}, \ref{prop:spc}, and \ref{prop:unary}, we are done.
\end{proof}

\begin{lemma}\label{lem:pure}
    Suppose $\rho$ is an irreducible, symmetric representation of $\SL$.  Then $\rho(\fs) = \tilde s$ or $i \cdot \tilde{s}$ for some real symmetric matrix $\tilde{s}$. 
\end{lemma}
\begin{proof}
    Denote $s \defeq \rho(\fs)$. Since $\rho$ is unitary and $s$ is symmetric, $s^{-1} = s^\dagger = \ol{s}$. Because $\fs^2$ is in the center of $\SL$, Schur's Lemma shows that $s^2 \in \BC \cdot\id$. Since $s^4 =\id$, $s^2 = \pm \id$ and  $\ol s = s^3$.  If $s^2 = \id$, then $\ol s = s$ and so $\tilde{s} \defeq s$ is real; otherwise, $(i \cdot s)^2 = \id$ and so $\tilde{s} \defeq -i \cdot s$ is real.
\end{proof}

\begin{cor}
    Every irreducible, congruence representation of $\SL$ is equivalent to a representation $\rho$ such that $\rho(\fs) = \tilde{s}$ or $i \cdot \tilde{s}$  for some real symmetric matrix $\tilde{s}$.
\end{cor}
\begin{proof}
    This follows immediately Theorem \ref{thm:main} and Lemma \ref{lem:pure}.
\end{proof}

\bibliographystyle{abbrv}
\bibliography{Reference}

\end{document}